\documentclass{amsart}
\usepackage{graphicx}
\usepackage{amsthm,amsmath,verbatim,amssymb}
\usepackage{caption}
\usepackage{color}
\usepackage{arydshln}
\newtheorem{lem}{Lemma}
 
\newtheorem{thm}{Theorem}
 
\newcommand{\CP}{\mathbb{CP}^1}
 
\newcommand{\tp}{two-point correlation }
\newcommand{\kahler}{K\"ahler }
\renewcommand{\d}{\partial}
\newcommand{\dbar}{\bar\partial}

\newcommand{\hn}{H^0(M,L^n)}

\newcommand{\e}{\mathbf{E}}

\newcommand{\C}{\mathbb{C}}
 
\begin{document}
\title[Two-point correlation]{Correlations between zeros and critical points of  random analytic functions }
\author[Renjie Feng]{Renjie Feng}
\address{Beijing International Center for Mathematical Research, Peking University, Beijing, China}
\email{renjie@math.pku.edu.cn}

\date{\today}
   \maketitle
   \begin{abstract}

 We study  the two-point correlation $K^m_n(z,w)$ between zeros and critical points of Gaussian random holomorphic sections $s_n$  over \kahler manifolds.  The critical points are points $\nabla_{h^n} s_n=0$ where $\nabla_{h^n}$ is the smooth Chern connection with respect to the Hermitian metric $h^n$ on line bundle $L^n$.  The main result is that  the rescaling limit of $K^m_n(z_0+\frac u{\sqrt n}, z_0+\frac v{\sqrt n})$ for any $z_0\in M$ is universal   as $n$ tends to infinity. In fact, the universal rescaling limit is the \tp  between zeros and critical points of Gaussian analytic functions for the Bargmann-Fock space of level $1$. Furthermore,  there is a 'repulsion' between zeros and critical points for the short range; and a 'neutrality' for the long range.   
   \end{abstract}

\section{Introduction}


In this article, we study the two-point correlation between critical points and zeros of random analytic functions and its generalization to   random holomorphic sections on \kahler manifolds. The famous Gauss-Lucas Theorem states that the holomorphic critical points of any polynomial of complex one variable are contained in the convex hull of its zeros.  This implies that some non-trivial correlations between zeros and critical points of
 random polynomials must exist.  It seems that the analogous properties should exist for random holomorphic sections on \kahler manifolds.  In \cite{F},  the author studied two conditional expectations on Riemann surfaces: the expected density of zeros of Gaussian random sections with a conditioning critical point and the expected density of critical points with a fixed zero.  It's proved that  both conditional densities have universal rescaling limits but the short  range behaviors are quite different:  there is a 'neutrality' between critical points and the conditioning zero while there is a 'repulsion' between zeros and the conditioning critical point.  In this paper, we  further study the \tp between zeros and critical points of Gaussian random holomorphic sections and its rescaling limit. The essential   difference to the Gauss-Lucas setting   is that the critical points are defined as zeros of the derivative of the smooth Chern connection $\nabla_{h^n}$ with respect to the Hermitian metric $h^n$ on the line bundle $L^n$ instead of
the holomorphic derivative $\frac{\partial}{\partial z}$ (a meromorphic connection). Hence, the two-point correlation should depend on the geometry, i.e., metrics defined on line bundles and \kahler manifolds. But we will show that
the rescaling limit of the \tp  is universal. In fact, the universal rescaling limit is the \tp between zeros and critical points of Gaussian analytic functions for the Bargmann-Fock space of level $1$. Such universal rescaling limit phenomenon was first proved by Bleher-Shiffman-Zelditch in \cite{BSZ1} for the \tp between zeros of Gaussian random holomorphic sections.  In this article, we will generalize Bleher-Shiffman-Zelditch's method to derive the universal rescaling limit of the \tp between zeros and critical points on \kahler manifolds.  We will show that the rescaling two-point correlation will tend to $0$ for the short  range and tend to a positive constant (which only depends on the dimension) for the long range. Roughly speaking,    there is a 'repulsion' between zeros and critical points for the short range and a 'neutrality' for the long range which means that zeros and critical points behave independently if they are far apart.

%
\subsection{Main results }
To state our results, we need to recall some basic  definitions of Gaussian random holomorphic sections of a line bundle (see \S \ref{bg}).
 We let $(L,h)\to (M,\omega)$ be a positive Hermitian holomorphic line bundle
over a compact \kahler manifold of complex $m$-dimensional. 
We denote
$H^0(M,L^n)$ as the space of global holomorphic sections of the $n$-th tensor power of $L$.
 The Hermitian metric $h$ will induce an inner product on $H^0(M, L^n)$ \eqref{innera}
and thus induces a Gaussian measure $d\gamma_{d_n}$ on $H^0(M, L^n)$, where $d_n$ is the dimension of  $H^0(M, L^n)$. A special case is when $M=\CP\cong S^2$ and $L=\mathcal O(1)$ the hyperplane line bundle, $H^0(\CP, \mathcal O(n))$ is the space of homogeneous polynomials of degree $n$. There is a classical Fubini-Study metric defined on the line bundle $(\mathcal O(1), h_{FS})\to (\CP, \omega_{FS})$ which will induce an inner product on $H^0(\CP, \mathcal O(n))$. Hence, it will induce a Gaussian measure on the space of 
homogeneous polynomials of degree $n$; the corresponding random polynomials are called Gaussian $SU(2)$ polynomials which are invariant under the rotation  on $S^2$, or equivalently, the $SU(2)$ action on $\CP$. 

Throughout the article, we assume our line bundle is polarized, i.e., $h=e^{-\phi}$, where $\phi$ is the smooth local \kahler potential such that $\omega=\d\dbar \phi$. 
Given a global holomorphic section $s_n$, we write $s_n=f_ne^{\otimes n}$ in a local coordinate patch where $f_n$ is a holomorphic function,  the Chern connection of $s_n$ is given by \cite{GH}
 \begin{equation}\label{dddddk}\nabla_{h^n}s_n=\sum_{i=1}^m\left(\frac{\partial f_n}{\partial z_i}-n\frac{\partial \phi}{\partial z_i} f_n\right) e^{\otimes n}\otimes dz_i. \end{equation}
 The critical points of holomorphic sections are points where $\nabla_{h^n}s_n=0$. Note that $\nabla_{h^n}s_n=0$ is  only a smooth equation instead a holomorphic equation since $\phi$ is smooth; furthermore, the solution depends on the geometry. This is the essential difference in our study between complex manifolds and the complex plane.  For example, let $M$ be a compact Riemann surface,  the total number of zeros of non-zero holomorphic sections of the positive holomorphic line bundle $L^n\to M$ is $c_1(L)n$ which is topologically invariant \cite{GH}, but the total number of critical points is not topological.   Deterministically, given a holomorphic section, one can
not tell how many critical points it has. But on average, it's proved in \cite{DSZ1, DSZ2} that the expected number of critical points has the asymptotics,  
$$\mathbf E(\# \mbox{of critical points defined by Chern connection})$$$$=\frac{5}{3}c_1(L)n+\frac{7}{9}(2g-2)+(\mbox{non-topological term})n^{-1}+\cdots,$$ 
where $g$ is the genus of the Riemann surface and the non-topological terms depend on the global geometry where the curvatures of the \kahler metric are involved. 
 Take Gaussian $SU(2)$ polynomials $p_n (z)$ for example, the holomorphic derivative $\frac{\partial p _n}{\partial z}=0$ always gives $n-1$
critical points; but the average number of  critical points defined by the smooth Chern derivative is asymptotic to $\frac{5}{3}n$ (recall $c_1(\mathcal O(1))=1$).


 We define the two-point correlation between zeros and critical point of Gaussian random sections with respect to $(H^0(M, L^n),d\gamma_{d_n})$ as \begin{equation}\label{firstdef} K^m_n(z,w) := \e_{(H^0(M, L^n),d\gamma_{d_n})}\left(\sum_{z: s_n(z)=0 }  \delta_z\otimes \sum_{w: \nabla_{h^n}s_n(w)=0 }\delta_w \right).\end{equation} 
 Note that given a  non-zero global holomorphic section, the zero set is an algebraic variety of codimension $1$ and the set of critical points is codimension $m$. 
Thus, $K^m_n(z,w)$ is a $(m+1,m+1)$-current on $M\times M$ in the sense of distribution, \begin{equation}\label{firstdef}\begin{split}&\int_{M\times M}\psi(z)\varphi(w)  K^m_n(z,w) \\ &=\e_{(H^0(M, L^n),d\gamma_{d_n})} \left(\int_{\{z: s_n(z)=0\} } \psi(z) \sum_{w: \nabla_{h^n}s_n(w)=0 } \varphi(w)\right)\end{split}\end{equation}
where $\psi$ is any smooth $(m-1,m-1)$-current on $M$ and $\varphi$ is a  smooth test function.




The purpose of the article is to study the typical spacing between zeros and critical points. We rescale the global expression of $K^m_n(z, w)$ by a factor $ {\sqrt n}$ at any fixed point $z_0\in M$, i.e., we enlarge the local geodesic ball by a factor $\sqrt n$. Note that $K^m_n(z, w)$ is 
a $(m+1,m+1)$-current depending on the Hermitian metric $h$ on the line bundle, but our main result claims that its rescaling limit is universal, 
\begin{thm}\label{main}
The rescaling  of the $(m+1,m+1)$-current of the \tp of zeros and critical points of Gaussian random sections with respect to $(\hn, d\gamma_n)$ has the following pointwise universal limit,
\begin{equation}\lim_{n\to \infty }K^m_n(z_0+\frac u{\sqrt n},z_0+\frac v{\sqrt n} )=K_{BF}^m(u,v)\frac {d\ell_u}\pi \wedge \frac{(d\ell_v)^m }{\pi^m m!},\end{equation}
where we denote $d\ell_z:= \frac i 2\sum_{j=1}^m dz_j\wedge d\bar{z_j}$ such that  $\frac{(d\ell_z)^m }{ m!}$ is the Lebesgue measure on $\mathbb C^m$. 
In fact,  $K_{BF}^m(u,v)$ is the two-point correlation function between zeros and critical points of Gaussian analytic functions of the Bargmann-Fock space of level $1$;  the explicit expression of $K_{BF}^m(u,v)$  is given by \eqref{higherdim}. Furthermore,   $K_{BF}^m(u,v)$ is a function of $|u-v|$ and it  admits the following pointwise limits, 
 \begin{equation}\label{timesss}\lim_{|u-v|\to 0}K_{BF}^m(u,v)=0\,\,\,\mbox{and}\,\,\, \lim_{|u-v|\to \infty}K_{BF}^m(u,v)=c_m.\end{equation}
 where $c_m$ is a constant only depending on the dimension, in particular, $c_1=\frac 53$.
 
\end{thm}
 To prove this, we will first derive a Kac-Rice type formula on \kahler manifolds. We will see that the \tp  can be expressed by the Bergman kernel and its derivatives up to order $4$.  It's well-known that the Bergman kernel on any \kahler manifold has a universal rescaling limit -- the Bergman kernel for the Bargmann-Fock space of level $1$. Hence,  the \tp   will admit a universal rescaling limit; the limit is actually the \tp  of Gaussian analytic functions of the Bargmann-Fock space of level $1$.

Theorem \ref{main} determines some local behaviors
between critical points and  zeros. 
 Intuitively, the rescaling limit of the \tp measures  the asymptotic probability of finding critical points and zeros in the small geodesic ball of radius of order $n^{-\frac{1}{2}}$.
 Roughly speaking, let's take the $1$-dimensional Riemann surfaces for example, the rescaling limit $K_{BF}^m(u,v)$ tending to $0$ as $|u-v|\to 0$ indicates that  it's unlikely to find
a zero and a critical point nearby simultaneously, i.e., there is a `repulsion` between zeros and critical points. The limit $K_{BF}^m(u,v)$ tending to $\frac 53$ as $|u-v|\to \infty$ indicates that zeros and critical points can not 'feel' each other for the long range, or equivalently, there is no correlation for the long range.  

A possible explanation for the 'repulsion' phenomenon is as follows. For the positive holomorphic line bundle,
it's well known that  the local minima of the  $h$-norm   $|s_n|_{h^n}$ are its zeros and the local maxima/saddle points of $|s_n|_{h^n}$ are obtained at the critical points  $\nabla_{h^n}s_n=0$ \cite{GH}.
Intuitively, at the zero of $|s_n|_{h^n}$,  $|s_n|_{h^n}$ is 'turning up' and it is very possible that it takes a while for $|s_n|_{h^n}$ to reach the local maxima/saddle, or equivalently, the process can not touch the local maxima/saddle immediately after it leaves $0$, which  implies that a 'repulsion' could occur between local minima and local maxima/saddle of $|s_n|_{h^n}$. This might explain that  a 'repulsion'  exists between zeros and and critical points of $s_n$. 

\subsection{Comparisons between Meromorphic and Chern connections}
As a  remark, the  two-point correlation between zeros and holomorphic critical points   has been studied recently in \cite{Hann2, Hann1, Hann3} for Gaussian $SU(2)$ polynomials on the complex plane $\mathbb C$. In fact, the Gaussian $SU(2)$ polynomials can be viewed as   meromorphic functions on $\CP$  and the holomorphic derivative  $\frac{\d}{\d z}$ can be viewed as a meromorphic connection on $\CP$ which has a pole at infinity.  In \cite{Hann2},  the two-point correlation function between zeros and the holomorphic critical points is derived by  the Poincar\'{e}-Lelong formula (but the author did not derive the rescaling limit). In  \cite{ Hann1, Hann3}, it is also proved that 
zeros and  critical points appear
in rigid pairs, to be more precise, given a zero, with high probability there is a unique critical point in the ball of radius of order $n^{-1}$ around the zero.  

The smooth Chern connection plays an important role in our results compared with   meromorphic connections.  As we show in this article,   the rescaling limit     $K_n(z+\frac u{\sqrt n},z+\frac v{\sqrt n})$ of the \tp between zeros and critical points (defined by the smooth Chern connection) is universal if we rescale the local domain by a factor $n^{-\frac 12}$, roughly speaking, this implies that the typical spacing between zeros and critical points is
$n^{-\frac 12}$.  Let  $K_n^{mero}(z ,w)$ be the   \tp between zeros  and critical points defined by a meromorphic connection  $\frac{\d s_n }{\d z}=0$ for Gaussian random holomorphic sections $s_n$. In  \cite{FSZ},  we show that   $K_n^{mero}(z+\frac u{\sqrt n},z+\frac v{\sqrt n})$ also admits a universal  limit.  The above two rescaling limits exist since the covariance kernel  of Gaussian random sections $s_n$ is the Bergman kernel (see \eqref{cov}) and the Bergman kernel has the universal rescaling limit $e^{z\bar w}$ (see \S  \ref{us}).  In fact, following the main idea in \cite{BSZ1}  and our proof  in  \S  \ref{us},  in order to derive the rescaling limit, it's enough to consider the Gaussian analytic function $$f(z)=\sum_{j=0}^\infty \frac {a_j}{\sqrt{j!}}z^j,$$
where $a_j$ are i.i.d. standard complex Gaussian random variables with mean $0$ and variance $1$.  Both the limits   $K_n(z+\frac u{\sqrt n},z+\frac v{\sqrt n})$   and $K_n^{mero}(z+\frac u{\sqrt n},z+\frac v{\sqrt n})$  are universal and obtained by the corresponding ones for $f(z)$. But the behaviors of these two rescaling limits of $K_n$ and $K_n^{mero}$ are quite different. First note that  the distribution of zeros of $f(z)$ is invariant under the translation and rotation of the complex plane \cite{HKPV}. Now recall \eqref{bfc}  of the smooth Chern connection under the Bargmann-Fock metric, let $$g(z):=\nabla_{h_{BF}} f(z)=\frac{\d f}{\d z}-\bar z f ,$$ then it's easy to see the distribution of zeros of $g(z)$ is also invariant under the rotation and translation (by computing the covariance kernel of $g(z)$). Hence,   the universal rescaling limit  $K_{BF}(u,v):=\lim_{n\to\infty} K_n(z+\frac u{\sqrt n},z+\frac v{\sqrt n})$ is actually a function in the form of $K(|u-v|)$, i.e., it's independent of the position of $z\in M$ and it's a   function depending only on the distance of $|u-v|$. But, let the meromorphic derivative $$g^{mero}(z):=\frac {\d f}{\d z},$$
then it's easy to show that the zero set of $g^{mero}$ is only rotation invariant but not translation invariant, and hence, the universal rescaling limit $K^{mero}_{BF}(u,v):=\lim_{n\to\infty} K^{mero}_n(z+\frac u{\sqrt n},z+\frac v{\sqrt n})$ should be  a function in the form of $K^{mero}(z, |u-v|)$, i.e.,  it's a   function depending  on the position $z$ and the distance of $|u-v|$.  


\bigskip

The article is organized as follows. In \S \ref{bg}, we will  first recall some basic concepts about the positive holomorphic line bundles and  \kahler manifolds, then we will define Gaussian random holomorphic sections. In \S \ref{thm122}, we will derive a Kac-Rice type formula for the \tp of zeros and critical points of Gaussian random holomorphic sections on any \kahler manifold. In \S \ref{us}, we will see that the \tp has a universal rescaling limit since the Bergman kernel does. Then we will derive the estimates \eqref{timesss}: we will prove such estimates for Riemann surfaces, then  sketch the proof for higher dimensions. 

\bigskip

\textbf{Acknowledgement} The author would like to thank Steve Zelditch for many helpful suggestions and corrections on the manuscript. 


\bigskip





 \section{Background}\label{bg}
 In this section, we will review some basic concepts and notations on Gaussian random holomorphic sections of
 positive holomorphic line bundles over \kahler manifolds.

\subsection{\kahler manifolds}
Let $(M,\omega)$ be a compact \kahler manifold of complex $m$-dimensional with the \kahler form \begin{equation}\omega=\frac{\sqrt{-1}}{2}\d\dbar\phi,\end{equation} where $\phi$ is the smooth local \kahler potential in a local coordinate patch $U\subset M$. Let $(L,h)\to (M,h)$ be a positive holomorphic line bundle such that the curvature of the Hermitian metric $h$
 \begin{equation}\Theta_h=-\frac{\sqrt{-1}}{2} \d\dbar \log h  \end{equation}
is a positive $(1,1)$ form \cite{GH}. Let $e$ be a local non-vanishing holomorphic section of $L$ over $U \subset M$ such that locally $L|_U\cong U\times\C$ and the pointwise $h$-norm of $e$ is $|e|_{h} = h(e, e)^{
1/2}$. Throughout the article, we  assume that the line bundle is polarized, i.e.,  \begin{equation}\Theta_h=\omega\,\,\, \mbox{or equivalently}\,\,\, |e|^2_h=h(e,e)=e^{-\phi}.\end{equation} Thus, $\frac \omega \pi$ is a de
Rham representative of the Chern class $c_1(L)$.  Let  \begin{equation}dV= \frac{\omega^m}{\pi^m m!}\end{equation} be the volume form. We assume that the total volume is $$\int_M dV=1.$$

We denote by $H^0(M,L^n)$ the space of global holomorphic sections of the $n$-th tensor power of $L$. Locally, we can write the global holomorphic section of $L^n$ as $s_n=f_ne^{\otimes n}$ where $f_n$ is some holomorphic function on $U$.
We denote the dimension of $H^0(M,L^n)$ by $d_n$. The Hermitian
metric $h$ induces a Hermitian metric $h
^n$ on $L^n$ as $|e^{ \otimes n}|_{h^n}=|e|_h^n$,  i.e., $|s_n|^2_{h^n}=|f_n|^2h^n(e^{\otimes n},e^{\otimes n})=|f_n|^2e^{-n\phi}$.

Now we can define an inner product on $H^0(M,L^n)$ as the following integration
\begin{equation}\label{innera}\langle s_{n,1}, s_{n,2}\rangle_{h^n}:=\int_M h^n( s_{n,1}, s_{n,2})dV=\int_M f_{n,1}\overline{f_{n,2}}e^{-n\phi}dV\end{equation}
for $s_{n,j}=f_{n,j}e^{\otimes n}\in H^0(M,L^n)$ with $j=1,2$.

The Chern connection $\nabla_{h^n}$ of the line bundle $(L^n, h^n)$ is the unique connection which is compatible with the Hermitian metrics $h^n$ and the holomorphic structure of complex manifolds \cite{GH}.  The smooth Chern connection  can be decomposed  into holomorphic and antiholomorphic parts as
\begin{equation}\label{de1}\nabla_{h^n}=\nabla'_{h^n}+\nabla''_{h^n},\end{equation} where in the local coordinate, they read \begin{equation}\label{de2}\nabla'_{h^n}=d_z+n\d\log h\,\,\,\mbox{and}\,\,\, \nabla''_{h^n}=d_{\bar z}.\end{equation} For the polarized line bundle with $h=e^{-\phi}$,  the Chern connection is 
\begin{equation}\label{dechern1}\nabla_{h^n}'s_n=\sum_{i=1}^m\left(\frac{\partial f_n}{\partial z_i}-n\frac{\partial \phi}{\partial z_i}f_n\right) e^{\otimes n}\otimes dz_i\,\,\,\mbox{and}\,\,\, \nabla''_{h^n}s_n=\sum_{i=1}^m \frac{\d f_n}{\d \bar z_i}e^{\otimes n}\otimes d\bar z_i\end{equation} 
for smooth sections $s_n=f_ne^{\otimes n}$  in the local coordinate. For the special case when $s_n$ is a global holomorphic section, we have \begin{equation}\nabla_{h^n}s_n=\nabla'_{h^n} s_n.\end{equation}

\subsection{\kahler normal coordinate}\label{kahlernormal}
Given a complex $m$-dimensional \kahler manifold $(L,h)\to(M,\omega)$, we freeze at a point $z_0$ as the origin of the coordinate patch and we can choose a \kahler normal coordinate $\{z_j\}$ as well as an adapted frame $e_L$ of the line bundle $L$ around $z_0$.  It is well-known that in terms of \kahler normal coordinates $\{z_j\}$,
the K\"ahler potential $\phi$ has the following expansion in the neighborhood of the origin $z_0$, \begin{equation}\label{kahler}\phi(z,\bar z)= \|z\|^2 -\frac 14 \sum R_{j\bar
kp\bar q}(z_0)z_j\bar z_{\bar k} z_p\bar z_{\bar q} +
O(\|z\|^5)\;.\end{equation}And thus,  \begin{equation}\label{adapted}\phi(z_0)=0,\,\, \d\phi(z_0)=0,\, \d^2\phi(z_0)=0,\,\,\d\dbar\phi(z_0)=1,\,\, \omega(z_0)=d\ell_z ,\end{equation}
where  $d\ell_z:= \frac i 2\sum_{j=1}^m dz_j\wedge d\bar{z_j}$.
In general, $\phi$ contains a
pluriharmonic term $f(z) +\overline{f(z)}$, but  a change of frame for $L$
eliminates that term up to fourth order. We refer 
to \S 3.1 in \cite{DSZ1} for more details.

An example on the \kahler normal coordinate and the adapted frame is the affine coordinate for the Fubini-Study metric of the hyperplane line bundle over the complex projective space $(\mathcal O(1), h_{FS})\to (\CP,\omega_{FS})$.  The \kahler form on $\CP$ is the Fubini-Study form.  In an affine coordinate, the \kahler form and the \kahler potential for the Fubini-Study metric are
\begin{equation}\label{fsadapt}
 \omega_{FS}=\frac{\sqrt{-1}}{2}\frac{dz\wedge d \bar z}{(1+|z|^2)^2},\,\,\,\phi_{FS}(z)=\log (1+|z|^2).
\end{equation} It's easy to check that $\phi_{FS}$ satisfies \eqref{adapted} and the affine coordinate is actually the \kahler normal coordinate at $z_0=0$. We equip
$\mathcal O(1)$ with its
Fubini-Study metric. In fact,  we can choose an adapted frame $e(z)$ such that
\begin{equation}|e(z)|^2_{h_{FS}}=e^{-\phi}=\frac 1{1+|z|^2}.\end{equation}





\subsection{Bergman kernels}\label{bggg}
The Bergman kernel is the orthogonal projection from the $L^2$-integral
sections to the holomorphic sections
\begin{equation}\label{projectionof}\Pi_n(z,w): L^2(M, L^n)\to \hn
\end{equation}
with respect to the inner product \eqref{innera}.
It has the following reproducing property
\begin{equation}\label{repo} \langle s_n(z),\Pi_n(z,w) \rangle_{h^n}=s_n(w),
\end{equation}
where $s_n \in \hn$ is a global holomorphic section.
Let $\{s_{n,1}, ..., s_{n,d_n}\}$
 be any orthonormal basis of  $\hn$ with respect to the inner product \eqref{innera}, then we have, 
\begin{equation}\label{bergman1} \Pi_n(z,w)=\sum_{j=1}^{d_n} s_{n,j}(z)\otimes \overline{s_{n,j}(w)}.
\end{equation}
We write $s_{n,j}=f_{n,j}e^{\otimes n}$ locally,  then we can rewrite \begin{equation}\Pi_n(z,w):=F_{n}(z,w) e^{\otimes n}(z)\otimes \overline{e^{\otimes n}(w)}\end{equation}
with the local function  \begin{equation}\label{defineF}F_n(z,w)=\sum_{j=1}^{d_n} f_{n,j}(z) \overline{f_{n,j}(w)},\end{equation}
where $F_n(z,w)$ is holomorphic in $z$ and anti-holomorphic in $w$. 

The  pointwise $h^n$-norm of the Bergman kernel has the following  Tian-Yau-Zelditch $C^\infty$-expansion on the diagonal \cite{ T,Y,Ze2},
\begin{equation}\label{full}
|\Pi_n(z,z)|_{h^{n}}=F_n(z,z)e^{-n\phi}= n^m(1+a_1(z)n^{-1}+a_2(z)n^{-2}+\cdots),
\end{equation}
where all terms $a_j$ are computable and they are polynomials of curvatures, in particular, $a_1$ is the scalar curvature of $\omega$.



Take the hyperplane
line bundle $\mathcal O(1)$ over the complex projective space $\CP$ for example, the global holomorphic sections of  $\mathcal O(1)$ are linear functions on $\C^2$ and hence the global holomorphic sections of
$L^
n = \mathcal O(n)$ are homogeneous polynomials of degree $n$. By choosing Fubini-Study metrics on $(\mathcal O(1), h_{FS})\to (\CP,\omega_{FS})$,  an orthonormal basis of $H^0(\CP,\mathcal O(n))$ under the inner product \eqref{innera} is given by
\begin{equation}\label{basis1}
\left\{\left(\sqrt{(n+1) {n\choose j}}z^j\right) e^{\otimes n}\right\}_{j=0}^{n}.
\end{equation}

Thus, the Bergman kernel for the Fubini-Study case is
\begin{equation}\label{fsberg}
F_n^{FS}(z,w)=(n+1) (1+z\bar w)^n.
\end{equation}

\subsection{Gaussian random fields}\label{Gaussianva}

Let's recall that  a complex  Gaussian measure on $\C^k$ is a measure of the
form
\begin{equation}\label{cxgauss}d\gamma_\Delta= \frac{e^{-z^*\Delta^{-1} z}}{\pi^{k}{\det\Delta}} dV_z\,,\end{equation}
where $dV_z$ denotes Lebesgue measure on $\C^k$ and $\Delta$ is a
positive definite Hermitian $k\times k$ matrix.  The  matrix
$\Delta$ is the covariance matrix.

The inner product \eqref{innera} induces a complex Gaussian probability measure $d\gamma_{d_n}$ on the space $\hn$ as,
\begin{equation}\label{standardgauss}d\gamma_{d_n}(s_n)=\frac{e^{-|a|^2}}{\pi^{d_n}}da,\,\,\,\,\, s_n=\sum_{j=1}^{d_n}a_j s_{n,j},\end{equation}
where $\{s_{n,1},..., s_{n,d_n}\}$
 is an orthonormal basis for $\hn$ and $\{a_1,..., a_{d_n}\}$ are i.i.d. standard complex Gaussian random
 variables with mean $0$ and variance $1$. 

Thus, by discarding the local frame,  the covariance kernel of the Gaussian random section $s_n$ is given by
\begin{equation}\label{cov}Cov(s_n(z),s_n(w))=F_n(z,w),\end{equation}
i.e., the Bergman kernel.

  \section{Kac-Rice type formula}\label{thm122}

In this section, we will derive a Kac-Rice type formula for the global expression of \tp of $K^m_n(z,w)$.
The formula may be derived from \cite{AT, BSZ1, DSZ1, DSZ2} but we take advantage of some simplifications to speed up the proof. We will only derive the formula on Riemann surfaces, then generalize naturally to higher dimensions. 

\subsection{Kac-Rice formula}\label{kacrice}
 We will prove the following   Kac-Rice type formula for the two-point correlation on Riemann surfaces,

\begin{lem}\label{joint1} On Riemann surfaces $(M,\omega)$, the $(2,2)$-current of the \tp of zeros and critical points of  holomorphic sections $s_n$ with respect to the Gaussian measure $d\gamma_{d_n}$ is

\begin{equation}  K^1_n(z,w)=\left(\pi^2 \int_{\mathbb C^3}p^n_{z,w}(0,0, \xi_1, \xi_2,\xi_3)  |\xi_1|^2 \left||\xi_2|^2- |\xi_3|^2\right| dV_\xi \right)dV(z)dV(w),\end{equation}
where  $dV_\xi$ is the Lebesgue measure on $\mathbb C^3$, $dV=\frac\omega\pi$ is the volume form on the Riemann suface and $p^n_{z,w}(x,y, \xi_1, \xi_2,\xi_3)$ is the joint density of Gaussian processes $(s_n(z), \nabla_{h^n}' s_n(w),\nabla'_{h^n} s_n(z),\nabla'_{h^n} \nabla'_{h^n} s_n(w),\nabla_{h^n}''\nabla_{h^n}' s_n(w))$
which can be expressed by the Bergman kernel and its Chern derivatives up to order $4$. 
\end{lem}
\begin{proof}
The strategy to get this formula is to find the local expression for the \tp  under the local coordinate, then we turn it to be global.

We denote the zero set $$\mathcal Z:=\left\{z\in M:\,\,s_n(z)=0\right\}$$ and the critical point set $$\mathcal C:=\left\{z\in M:\,\,\nabla_{h^n}s_n(z)=0\right\}.$$

In the local coordinate $U\cong\C$ and a local trivialization of $L$, we write Gaussian random holomorphic sections   as $s_n=f_ne^{\otimes n}$ where $f_n$ is a holomorphic function, we denote locally $$g_n(z):=\nabla_{h^n} s_n=(\frac{\d f_n}{\d z}-n\frac{\d\phi}{\d z} f_n)e^{\otimes n}\otimes dz.$$ Then the set of  critical points of $s_n$ is the same as zeros of $g_n=0$ (recall definition of Chern connection \eqref{dechern1}). We denote locally
$\mathcal Z_U:=\left\{z\in \C:\,\,  f_n(z)=0\right\}$  as zeros in $U$;  denote $\mathcal C_U:=\left\{z\in \C:\,\,  g_n(z)=0\right\}$ as the set of critical points in $U$.
 


By definition of the delta function, for any smooth test functions $\psi$ and $\varphi$ on $M$, we have locally,
\begin{align*}&\langle \sum_{z\in \mathcal Z_U}\delta_{z}\otimes \sum_{w\in \mathcal C_U}\delta_{w},\psi\otimes \varphi \rangle\\
=& \sum_{z:\,\,\, f_n(z)=0}\psi(z) \sum_{w:\,\,\, g_n(w)=0}\varphi(w)\\=& \int_{\C\times\C} \delta_0(f_n(z))\delta_0(g_n(w))\psi(z) \varphi(w)df_n(z)\wedge d\bar{f}_n(z)\wedge dg_n(w)\wedge d\bar{g}_n(w)\\=&
\int_{\C\times \C}  \delta_0(f_n)\delta_0(g_n)\psi(z) \varphi(w)  |\frac{\d f_n}{\d z}|^2 \left| |\frac{\d g_n}{\d w}|^2-|\frac{\partial g_n}{\d\bar w}|^2\right| d\ell_z d\ell_w,
\end{align*} where we denote $d\ell $ as the Lebesgue measure on $\C$.

Now we can turn the above  integral to be global by the following observations (by discarding the local frame).  Let's first recall the decomposition of the Chern connection $\nabla_{h^n}=\nabla'_{h^n}+\nabla_{h^n}''$ with $\nabla_{h^n}' =d_z-n\frac{\d \phi}{\d z}\,\,\,\mbox{and}\,\,\, \nabla_{h^n}'' =d_{\bar z}$.     At $z_0$, the zero of the holomorphic section $s_n$ where $f_n(z_0)=0$, we have $$ \nabla_{h^n}' f_n(z_0)=\frac{\d f_n}{\d z}(z_0)-n\frac{\d \phi}{\d z}f_n(z_0)=\frac{\d f_n}{\d z}(z_0).$$  At the critical point $w_0$ with $g_n(w_0)=0$,  recall the definition $g_n:=\nabla_{h^n} f_n= \nabla'_{h^n} f_n$, by taking derivatives on both sides, we have $$\nabla_{h^n}' \nabla_{h^n}' f_n(w_0)=\nabla_{h^n}' g_n(w_0)=\frac{\d g_n}{\d w}(w_0)-n\frac{\d \phi}{\d w}g_n(w_0)=\frac{\d g_n}{\d w}(w_0)$$ and $$\nabla_{h^n}'' \nabla_{h^n}' f_n(w_0)=\nabla_{h^n}'' g_n(w_0)= \frac{\d g_n}{\d\bar  w}(w_0).$$ 

Hence,  the global expression for the above integration is, 
$$\int_{M\times M}  \delta_0(s_n(z))\delta_0(\nabla_{h^n}' s_n(w))\psi(z) \varphi(w) $$$$\times |\nabla_{h^n}' s(z)|^2 \left| |\nabla_{h^n}' \nabla_{h^n}' s_n(w)|^2- |\nabla_{h^n}''\nabla_{h^n}' s_n(w)|^2\right| \omega_z\wedge \omega_w.$$

By taking the expectation on both sides,  we have globally, \begin{align*}&\e\langle \sum_{z\in \mathcal Z}\delta_{z}\otimes \sum_{w\in \mathcal C}\delta_{w},\psi\otimes \varphi \rangle\\
=&
\int_{M\times M} \psi(z) \varphi(w)\left(\int_{\mathbb C^3}p^n_{z,w}(0,0, \xi_1, \xi_2,\xi_3)  |\xi_1|^2 \left||\xi_2|^2-|\xi_3|^2\right| dV_\xi\right)\omega_z\wedge \omega_w,
\end{align*} 
where  $dV_\xi$ is the Lebesgue measure on $\mathbb C^3$ and $p^n_{z,w}(x,y, \xi_1, \xi_2,\xi_3)$ is the joint density of Gaussian processes $(s_n(z), \nabla'_{h^n} s_n(w),\nabla'_{h^n} s_n(z),\nabla_{h^n}' \nabla_{h^n}' s_n(w),\nabla_{h^n}''\nabla_{h^n}' s_n(w))$.

Hence, the \tp in the case of Riemann surfaces is given by
\begin{align*}&\e \left(\sum_{z\in \mathcal Z }\delta_{z}\otimes \sum_{w\in \mathcal C }\delta_{w}\right):= K^1_n(z,w)\\
&=\left(\int_{\mathbb C^3}p^n_{z,w}(0,0, \xi_1, \xi_2,\xi_3)  |\xi_1|^2 \left||\xi_2|^2- |\xi_3|^2\right| dV_\xi \right)\omega(z)\wedge \omega(w).\end{align*} 
The extra factor $\pi^2$ in Lemma \ref{joint1} appears since we  define the volume form $dV:=\frac \omega \pi$. 
For the last statement, note that the covariance kernel of the Gaussian process $s_n$ is the Bergman kernel, i.e., $\e(s_n(z)\overline{s_n(w)})=F_n(z,w)$ (see  \eqref{cov} ), hence the covariance matrix of Gaussian processes $(s_n(z), \nabla_{h^n}' s_n(w),\nabla' s_n(z),\nabla_{h^n}' \nabla_{h^n}' s_n(w),\nabla_{h^n}''\nabla_{h^n}' s_n(w))$ can be expressed by the Bergman kernel and its Chern derivatives up to order $4$ (see \cite{AT}), this completes the proof of Lemma \ref{joint1}.
\end{proof}
 \subsection{Higher dimensions}
For higher dimensions, given a smooth section $s_n=f_ne^{\otimes n}$, the Chern connection has the decomposition (see \eqref{dechern1}), 
  $$\nabla_{h^n}'s_n=\sum_{i=1}^m\left(\frac{\partial f_n}{\partial z_i}-n\frac{\partial \phi}{\partial z_i} f_n\right) e^{\otimes n}\otimes dz_i, \,\,\,\,\nabla_{h^n}''s_n=\sum_{i=1}^m \frac{\partial f_n}{\partial \bar z_i}  e^{\otimes n}\otimes d\bar z_i. $$
 
We further rewrite $\nabla_{h^n}'$ and $\nabla_{h^n}''$ as, 
 $$\nabla_{h^n}' =\sum_{i=1}^m\nabla_{h^n,i}', \,\,\,\,\nabla_{h^n}'' =\sum_{i=1}^m\nabla_{h^n,i}'',  $$
where we define
$$\nabla_{h^n,i}'s_n= \left(\frac{\partial f_n}{\partial z_i}-n\frac{\partial \phi}{\partial z_i} f_n\right) e^{\otimes n}\otimes dz_i,\,\,\,\,\nabla_{h^n,i}''s_n= \frac{\partial f_n}{\partial \bar z_i}  e^{\otimes n}\otimes d\bar z_i.$$

Following the computations in \S\ref{kacrice} (or \S 2 in \cite{DSZ1}), we have, 
 \begin{lem}\label{joint1222} The $(m+1,m+1)$-current of the \tp of zeros and critical points of Gaussian holomorphic sections on any compact \kahler manifold of complex $m$-dimensional is
\begin{align*}  K^m_{n}(z,w)&=\left(\pi^{m+1} \int_{\mathbb C^{m^2+2m}}p^n_{z,w}(0,0, \xi, H_1,H_2)  \|\xi\|^2 \left|\det(H_1^*H_1 - H_2^*H_2)\right| dV_\xi dV_H \right)\\
&\times \frac{\omega(z)}{\pi}\wedge \frac{\omega^m(w)}{\pi^m m!},\end{align*}
where  $d	V_\xi$ and $dV_{H}$ are Lebesgue measures on $\xi\in \mathbb C^m$ and $(H_1, H_2)\in \mathbb C^{m(m+1)}$ where $H_1$ and $H_2$ are two symmetric $m\times m$ matrices, $\|\xi\|^2$ is the norm square of the vector $\xi$  and $p^n_{z,w}(x,y, \xi, H_1,H_2)$ is the joint density of Gaussian processes $\left(s_n(z), (\nabla_{h^n,i}' s_n(w))_{i=1}^m,(\nabla'_{h^n,i} s_n(z))_{i=1}^m,(\nabla'_{h^n,i} \nabla'_{h^n,j} s_n(w))_{i,j},(\nabla_{h^n,i}''\nabla_{h^n,j}' s_n(w))_{i,j}\right)$ with $1\leq i\leq  j\leq m$.
\end{lem}

\section{Universality and scaling}\label{us}

It can be tell from Lemmas \ref{joint1} and \ref{joint1222} that the \tp is expressed by the Bergman kernel and its Chern derivatives, hence, the rescaling limits of \tp  depend only on the rescaling limits of Bergman kernel and its Chern derivatives. Our plan is the following. In subsection \S \ref{BF}, we will describe the Bergman kernel for the Bargmann-Fock space. This model case provides the universal rescaling limit for Bergman kernels on any \kahler manifold. Actually, the universal rescaling limit is the Bergman kernel of the Bargmann-Fock space of level $1$. In  \S \ref{proofmain}, we will prove Theorem \ref{main} for Riemann surfaces. In \S \ref{higherdims}, we will sketch the proof for the higher dimensions. 

\subsection{Bargmann-Fock space}\label{BF}
The Bargmann-Fock space is the space of entire functions on
$\mathbb C^m$ which are $L^2$-integral with respect to the Bargmann-Fock metric. To be more precise, let's
take the trivial line bundle $(L:=\mathbb C\times \mathbb C^m,h_{BF}(z))$ over $(\mathbb C^m,  \pi^{-m}dV_z)$ with the Hermitian Bargmann-Fock metric $h_{BF}(z):=e^{-\|z\|^2}$ and the Lebesgue measure $dV_z$ on $\mathbb C^m$,  here we denote $\|z\|^2=|z_1|^2+\cdots+|z_m|^m$.  The line bundle is trivial, we may use the frame $e_U=e_{\mathbb C}=1$. By \eqref{dechern1}, the Chern connection in this case  is given by
\begin{equation}\label{bfc}\nabla_{h_{BF}}=\nabla_{h_{BF}}'+\nabla_{h_{BF}}''\,\,\,\mbox{with}\,\,\,\nabla_{h_{BF}}'=\sum (d_{z_j}-\bar z_j) \,\,\, \mbox{and} \,\,\,\nabla_{h_{BF}}''=\sum d_{\bar z_j}.\end{equation}

 We raise the power of the line bundle to $L^{\otimes n}$ and define the Bargmann-Fock space $\mathcal H(\mathbb C^m, \pi^{-m}e^{-n\|z\|^2}dV_z)$ of level $n$ to be the space of $L^2$-entire functions with respect to the inner product (recall \eqref{innera})
\begin{equation}\langle f, g\rangle_{h_{BF}^
n}= \int_{\mathbb C^m} f\bar g e^{-n\|z\|^2}\pi^{-m}dV_z.\end{equation}
 The Bargmann-Fock space is a Hilbert space  and the orthonormal basis is given by monomials 
\begin{equation}\label{timem}\left\{ \frac{z^\alpha}{\sqrt{\frac{\alpha !}{n^{m+|\alpha|}}}},\,\,\,\, \alpha\in\mathbb Z_{+}^m \right\}, \end{equation}
where we denote $z^\alpha=z_1^{\alpha_1}\cdots z_m^{\alpha_m} $ and $|\alpha|=|\alpha_1|+\cdots+|\alpha_m|$.

 Then  the Bergman kernel off the
diagonal for the Bargmann-Fock space of level $n$ is (recall \eqref{defineF}), 
\begin{equation}F_n^{BF}(z,w)= \sum_{\alpha\in \mathbb Z_{+}^m}\frac{z^\alpha \bar w^\alpha}{ \frac{\alpha !}{n^{m+|\alpha|}}}=n^m\sum_{\alpha\in \mathbb Z_{+}^m}\frac{n ^{|\alpha|}  z^\alpha \bar w ^\alpha}{ \alpha ! }=n^m e^{nz\cdot \bar w}\end{equation}
where $z\cdot \bar w=z_1\bar w_1+\cdots +z_m\bar w_m$.

The following asymptotic expansion is proved in \cite{BSZ1} which states that the Bergman kernel admits a universal rescaling limit on any \kahler manifold. 
Let $z_0 \in M$ and choose \kahler normal coordinates in a neighborhood of $z_0$  and adapted frame $e_L$, the Bergman kernel admits the full  expansion,
\begin{equation}\begin{split}
n^{-m}F_n(z_0+\frac u{\sqrt n}, z_0+\frac v{\sqrt n})&=e^{u \cdot \bar v}+O(n^{-1/2})\\&=F_1^{BF}(u,v)+O(n^{-1/2}),
\end{split}\end{equation}
where $F_1^{BF}(u,v)$ is the Bergman kernel for the Bargmann-Fock space of level $1$.

The proof of this asymptotic expansion is based on  Boutet de Monvel-Sjostrand parametrix 
construction and the stationary phase method. As a remark, the rescaling limits of Chern derivatives of the Bergman kernel are also universal by taking Chern derivatives on both sides  of the above full expansion \cite{BSZ1}. 

An example to illustrate this is the Bergman kernel for the hyperplane line bundle over complex projective space. Recall \eqref{fsberg} of Bergman kernel  for the Fubini-Study metric $(\mathcal O(n),h_{FS})\to (\CP,\omega_{FS})$,  by choosing the \kahler normal coordinate at $z_0$, 
the rescaling limit of the Bergman kernel satisfies the following pointwise limit,
\begin{equation}\lim_{n\to\infty} n^{-1} F^{FS}_n(z_0+\frac u{\sqrt n}, z_0+\frac v{\sqrt n})=\lim_{n\to\infty}(1+\frac{u \cdot \bar v}n)^n= e^{u   \bar v}.\end{equation}

\subsection{Proof of Theorem \ref{main} for Riemann surfaces}\label{proofmain}
Let's first derive the universal rescaling limit of the \tp between zeros and critical points on Riemann surfaces. As proved in Lemma \ref{joint1}, it's equivalent to derive the universal rescaling limit of
the joint density.   This is workable since the joint density is expressed by the Bergman kernel and the Bergman kernel has the universal rescaling limit; furthermore, the limit is achieved  by the Bargmann-Fock space of level $1$. Hence, following the main idea in \cite{BSZ1}, to prove the main result for Riemann surfaces,   it's enough to consider the following Gaussian analytic functions,
\begin{equation}\label{analytic}f(z)=\sum_{j=0}^\infty \frac {a_j}{\sqrt{j!}}z^j,\end{equation}
where $a_j$ are i.i.d. standard complex Gaussian random variables with mean $0$ and variance $1$ and \begin{equation}\left\{\frac {z^j}{\sqrt{j!}}\right\}_{j=0}^\infty\end{equation} is an orthonormal basis of the Bargmann-Fock space $\mathcal H(\mathbb C, \pi^{-1}e^{-|z|^2}dV_z)$ (recall \eqref{timem}). The \tp between zeros and critical points (defined by Chern connection \eqref{bfc}) of $f(z)$ is the rescaling limit of \tp of Gaussian random holomorphic sections on any \kahler manifold because the covariance kernel of $f(z)$ is 
\begin{equation}cov(f(z), f(w))=F_1^{BF}(z,w)=e^{z \bar w},\end{equation}
i.e.,  the Bergman kernel of the Bargmann-Fock space of level $1$.

Because of the universality of Bergman kernels, we have the following 

\begin{lem}\label{lemma2}
On Riemann surfaces, the    (2,2)-current of the \tp of Gaussian random holomorphic sections admits the following pointwise universal limit
$$ \lim_{n\to\infty}K^1_n(z_0+\frac u{\sqrt n}, z_0+\frac v{\sqrt n})=K_{BF}^{1}(u,v),$$
where $K_{BF}^{1}(u,v)$ is the \tp between zeros and critical points of the Gaussian random analytic function $f(z)$ defined in \eqref{analytic}.

\end{lem}

We refer to  \cite{BSZ1} for more details of this lemma.  This lemma completes the first part of our main Theorem \ref{main} for Riemann surfaces. 
In the followings, let's  derive the formula for $K_{BF}^{1}(u,v)$ and estimate  $K^1_{BF}(u,v)$
  as $|u-v|$ tends to $0$ and $\infty$. 
\subsubsection{Covariance matrix}
By Lemma \ref{joint1}, the \tp between  zeros and critical points of the Gaussian random analytic function $f(z)$ is
\begin{equation}\label{kbf}K_{BF}^{1}(u,v)= \left(\pi^2\int_{\mathbb C^3}p_{u,v}(0,0, \xi_1, \xi_2,\xi_3)  |\xi_1|^2 \left||\xi_2|^2- |\xi_3|^2\right|  dV_\xi\right)\frac {d\ell_u}\pi\wedge \frac {d\ell_v}\pi ,\end{equation}
 where  $p_{z,w}(0,0, \xi_1, \xi_2,\xi_3)$ is the joint density of  Gaussian processes 
 $(f(u), \nabla_{h_{BF}}' f(v),$
 $\nabla'_{h_{BF}} f(u),\nabla_{h_{BF}}' \nabla_{h_{BF}}' f(v),\nabla_{h_{BF}}''\nabla_{h_{BF}}' f(v))$ and we denote $d\ell_z$ as the Lebesgue measure on $\mathbb C$.  
 
By definition of the Chern connection for the Bargmann-Fock metric \eqref{bfc}, we have, 
$ \nabla_{h_{BF}}' f(z)=\frac{\d f}{\d z}-\bar z f$, $ \nabla_{h_{BF}}' \nabla_{h_{BF}}'  f(z)=\frac{\d ^2 f}{\d z^2}-2\bar z\frac{\d f}{\d z}+\bar z^2 f$ and 
$ \nabla_{h_{BF}}''\nabla_{h_{BF}}'  f=-f$.
 For such Gaussian processes with covariance kernel $\e (f(u)f(v))=e^{u\cdot \bar v}$,  the covariance matrix is given by \cite{AT},

\begin{equation} \Delta=\begin{pmatrix}A& B\\ B^*& C
    \end{pmatrix}_{5\times 5},\end{equation}

   where $$A=\begin{pmatrix} e^{|u|^2}& (u-v)e^{u \bar v}\\
      (\bar u-\bar v)e^{v\bar u} &e^{|v|^2} \end{pmatrix},$$

    $$B=\begin{pmatrix} 0&(u-v)^2 e^{u\bar v} & -e^{u\bar v}  \\ 
     (1+\bar u v+\bar v u-|u|^2-|v|^2) e^{v \bar u}&0 &0
    \end{pmatrix}$$
   and 
   \begin{align*} &C=\\ 
  & \begin{pmatrix} e^{|u|^2} &  ( u- v)(\bar u v+\bar v u+2-|v|^2-|u|^2)e^{u\bar  v}  &(\bar u-\bar v)e^{u\bar v}\\ 
    (\bar u-\bar v)(\bar u v+\bar v u+2-|v|^2-|u|^2)e^{\bar u v}&2e^{|v|^2}&0\\
  ( u-  v)e^{\bar u v}&0 &e^{|v|^2}\end{pmatrix}.\end{align*}
  
Given the covariance matrix, by elementary matrix computations, the joint density   in Lemma \ref{joint1} can be further simplified as \cite{AT},

 \begin{equation}\label{denty}p_{u,v}(0,\xi) =\frac 1{\pi^{5} }\frac 1{ \det A \det \Lambda}\exp\left\{- \xi^* \Lambda^{-1} \xi \right\} ,
\end{equation}
where

\begin{equation}\label{matrixd}\Lambda=C- B^* A^{-1} B\end{equation}
 is a positive symmetric matrix. 
 
We have the following observations to simplify our computations. Since the Bergman kernel for the Bargmann-Fock space is invariant with respect to unitary transformations and equivariant with respect to
 translations, hence zeros of Gaussian analytic functions $f(z)$ are also invariant with respect to the group of isometric translations, i.e., unitary transformations and translations of $\mathbb C$ \cite{HKPV}.  By computing the covariance kernel of $\nabla'_{h_{BF}} f$, we can  prove that   critical points are also rotation and translation invariant. Hence,  the \tp of the Gaussian analytic function  is a function depending only on the distance $r:=|u-v|$. Without loss of generosity, we take 
$u=r$ and $v=0$, then, 

$$A=\begin{pmatrix} e^{r^2}&r\\
     r &1\end{pmatrix},$$

    $$B=\begin{pmatrix} 0&r^2 & -1 \\ 
     1-r^2&0 &0
    \end{pmatrix}$$
   and 
   \begin{align*} &C=
  \begin{pmatrix} e^{r^2} &  2r-r^3 &r\\ 
  2r-r^3&2&0\\
 r&0 &1\end{pmatrix}.\end{align*}
  \subsubsection{Short range behavior}
Let's first derive the short range behavior of the \tp of Gaussian analytic functions as $r\to 0$.  It's easy to get $$\lim_{r\to 0}\det A=1\,\,\,\mbox{and}\,\,\, \lim_{r\to 0}\Lambda=  \begin{pmatrix} 0 & 0 &0\\ 
  0&2&0\\
 0&0 &0\end{pmatrix}.$$
 

 Let $$P=\begin{pmatrix} 1&0 & 0 \\ 
     0 &0 &0 \\ 0&0&0
    \end{pmatrix} \,\,\,\mbox{and}  \,\,\,Q=\begin{pmatrix} 0&0 & 0 \\ 
     0 &1 &0 \\ 0&0&-1
    \end{pmatrix}.$$ If we combine \eqref{denty}, we can fruther rewrite the $(2,2)$-current \eqref{kbf} as, 
\begin{equation}\label{dkdkk}K_{BF}^{1}(u,v):=\tilde  K_{BF}^{1}(u,v)\frac {d\ell_u}\pi\wedge \frac {d\ell_v}\pi\end{equation}
 where we denote
\begin{equation}\label{dkd} \tilde K_{BF}^1(u,v)= \frac1{\pi^3 \det A }\int_{\mathbb C^3} \frac {e^{- \xi^* \Lambda^{-1} \xi}}{ \det \Lambda} (\xi^*P\xi) \left|\xi^*Q\xi \right|  dV_\xi.\end{equation}
 as the \tp function.

Now we change variable $\xi\to  \Lambda^{-\frac 12}\xi$ to get, 

$$ \tilde K_{BF}^{1}(u,v)= \frac1{\pi^3 \det A }\int_{\mathbb C^3}  e^{-  \| \xi\|^2}  (\xi^*  \Lambda^{\frac 12}P  \Lambda^{\frac 12}\xi) \left|\xi^*  \Lambda^{\frac 12}Q  \Lambda^{\frac 12}\xi \right|  dV_\xi .$$
We observe that $\Lambda^{\frac 12}P  \Lambda^{\frac 12}$ can be uniformly bounded for $r$  small enough, thus we can change the order of the limit   $r\to 0$ and the integration of $\xi$. It's easy to see $$\lim_{r\to 0}\Lambda^{\frac 12}P  \Lambda^{\frac 12}=0,$$ hence,  
 $$\lim_{u\to v}\tilde K_{BF}^{1}(u,v)=0.$$ This proves that there is a 'repulsion' between zeros and critical points of  Gaussian analytic functions. 

\subsubsection{Long range behavior}\label{longrange}
Now we study the long range behavior of the \tp of Gaussian analytic functions as  $r\to \infty$.



As $r\to \infty$, we can derive the following estimate, 
 
   \begin{align*} &\Lambda=
  \begin{pmatrix} e^{r^2}-(r^2-1)^2 &  2r-r^3 &r\\ 
  2r-r^3&2&0\\
 r&0 &1\end{pmatrix}+O(r^{-\infty}).\end{align*}
Hence, the square root of $\Lambda$ has the following estimate as $r\to \infty$,
  \begin{align*} &\Lambda^{\frac12}=
  \begin{pmatrix} e^{r^2/2} &  0&0\\ 
  0&\sqrt 2&0\\
  0&0&1  \end{pmatrix}+O(r^{-\infty}).\end{align*}

  Thus, up to $O(r^{-\infty})
$ which is negligible, we have, 
\begin {align*}\tilde K_{BF}^{1}(u,v)
= \frac{e^{r^2}}{\pi^3 \det A }\int_{\mathbb C^3}  e^{-  \| \xi\|^2}  |\xi_1|^2 \left|  2|\xi_2|^2- |\xi_1|^2 \right|  dV_\xi+O(r^{-\infty}).\end{align*}

 
 

Note $\det A=e^{r^2}-r^2$, thus as $r\to \infty$, 
 
 $$\frac {e^{r^2}}{\det A}=1+O(r^{-\infty}).$$
 Hence, 
 
 \begin{align*}\tilde K_{BF}^{1}(u,v)&= \frac1{\pi^3  }\int_{\mathbb C^3}  e^{-  \| \xi\|^2}  |\xi_1|^2 \left|  2|\xi_2|^2- |\xi_1|^2 \right|  dV_\xi  +O(r^{-\infty})\\
 &=\frac 5 3 +O(r^{-\infty}).\end{align*}
This verifies that there is no correlation between zeros and critical points for Gaussian analytic functions for the long range, roughly speaking,  zeros and critical points behave independently if they are far apart. Hence, we complete our main Theorem \ref{main} for Riemann surfaces.



\subsection{Higher dimensions}\label{higherdims}
Now let's sketch the proof of Theorem \ref{main} for higher dimensional \kahler manifolds.   Again, it's enough to consider the following Gaussian analytic functions on $\mathbb C^m$
$$f(z)=\sum_{\alpha\in \mathbb Z^m_+}a_{\alpha}\frac{z^\alpha}{\sqrt{\alpha!}}$$ where
$a_\alpha$ are standard complex Gaussian random variables so that the covariance kernel 
$$Cov(f(z), f(w))=e^{z\cdot \bar w}.$$

We apply Lemma \ref{joint1222} to $\mathbb C^m$ with the Bargmann-Fock metric, the \tp between zeros and critical points for $f(z)$ is, 
$$ K^m_{BF}(u,v):=\tilde K^m_{BF}(u,v)  \frac{d\ell_u}{\pi}\wedge \frac{(d\ell_v)^m }{\pi^m m!},$$
where we denote   $d\ell_z:= \frac i 2\sum_{j=1}^m dz_j\wedge d\bar{z_j}$ such that  $\frac{(d\ell_z)^m }{ m!}$ is the Lebesgue measure on $\mathbb C^m$ and we denote
\begin{align*}  \tilde K^m_{BF}(u,v)&= \pi^{m+1} \int_{\mathbb C^{m^2+2m}}p^n_{u,v}(0, 0, \xi, H_1,H_2)  \|\xi\|^2 \left|\det(H_1^*H_1 - H_2^*H_2) \right|\\ 
& \times  dV_\xi dV_{H_1}dV_{H_2}    
\end{align*} as the \tp function. Here, $p^n_{u,v}(x,y, \xi, H_1,H_2)$ is the joint density of Gaussian processes  $(f(u), \nabla_{h_{BF},i}' f(v),$
 $\nabla'_{h_{BF},i} f(u),\nabla_{h_{BF},i}' \nabla_{h_{BF},j}' f(v),$$\nabla_{h_{BF},i}''\nabla_{h_{BF},j}' f(v))$ with $1\leq i \leq  j\leq m$.  
 
 By   definition of the Chern connection \eqref{bfc}, we have  $\nabla_{h_{BF},i}'  f= \frac{\d  f}{\d z_i } -\bar z_i f$, $\nabla_{h_{BF},j}'\nabla_{h_{BF},i}'  f= \frac{\d^2 f}{\d z_i\d z_j}-\bar z_i\frac{\d f}{\d z_j}-\bar z_j\frac{\d f}{\d z_i}+\bar z_i\bar z_j f$ and 
$ \nabla_{h_{BF},j}''\nabla_{h_{BF},i}'  f=-\delta_i^j f$. Both zeros and critical points of Gaussian analytic functions $f(z)$ are  rotation and translation invariant,  hence,  \tp is again a function depending only on the distance $r:=|u-v|$. Without loss of generosity, we take 
$u=(r,0,\cdots, 0)$ and $v=(0,\cdots, 0)$. The Gaussian processes are further simplified to be $$\left(f(u),\frac {\d f(v)}{\d v_i} 
  , \frac{ \d f(u)}{\d u_i}-\delta_1^irf(u),    \frac{\d^2 f}{\d v_i\d v_j}  , -\delta_i^j f(v)  \right) $$
evaluated at the point $u=(r,0,\cdots, 0)$ and $v=(0,\cdots, 0)$ where $1\leq i\leq j \leq m$.

Note that the last element is $-\delta_i^j f(v)$, this implies that  the dimension of the above Gaussian processes can be reduced when $u=(r,0,\cdots, 0)$ and $v=(0,\cdots, 0)$. Hence, the \tp can be further simplified to be \begin{equation}\label{higherdim}
\begin{split}
  \tilde K^m_{BF}(u,v)&= \pi^{m+1} \int_{\mathbb C^{\frac{(m+1)(m+2)}{2}}} p^n_{u,v}(0,0,  \xi, H_1,\eta)  \|\xi\|^2 \left|\det(H_1^*H_1 -|\eta|^2 I) \right|\\&\times dV_\xi dV_{H_1}dV_{\eta}    \end{split}
 \end{equation}
evaluated at $u=(r,0,\cdots, 0)$ and $v=(0,\cdots, 0)$, where $dV_\eta$ is the Lebesgue measure on $\C$.
Here, $p^n_{u,v}(x,y, \xi, H_1,\eta)$ is the joint density of Gaussian processes   
 $$\left(f(u),\frac {\d f(v)}{\d v_i} 
  , \frac{ \d f(u)}{\d u_i}-\delta_1^irf(u),    \frac{\d^2 f}{\d v_i\d v_j}  ,  -f(v)  \right) $$
evaluated at the point $u=(r,0,\cdots, 0)$ and $v=(0,\cdots, 0)$.

To compute the covariance matrix, following identities \eqref{denty}\eqref{matrixd}, we have,
 \begin{align*} A&=\begin{pmatrix} \e f(u)\overline { f(u)} & \e f(u)\overline{\frac {\d f(v)}{\d v_i}}   \\ 
   \e \overline{f(u)} \frac {\d f(v)}{\d v_i}&  \e \frac {\d f(v)}{\d v_i}\overline{\frac {\d f(v)}{\d v_j}}   \end{pmatrix}_{|_{u=(r,0,\cdots, 0),v=(0,\cdots, 0)}}\\
   &=\begin{pmatrix}e^{r^2} & r& 0 &\cdots &0 \\ 
  r& 1&0&\cdots&0 \\
  0&0&1&\cdots&0\\
  \vdots&\vdots &&&\vdots\\
  0 &0&0 &\cdots& 1 \end{pmatrix}_{(m+1)\times (m+1)}.\end{align*}
Following the same computations, we have, 
\[ B=\left(\begin{array}{cccc|ccccc}
  0& 0& \cdots &0 & r^2 &0 &\cdots& 0& -1\\\hline
1-r^2& 0& \cdots &   \\
  0&1&\cdots&   \\
   \vdots&  \vdots & \vdots  &\vdots  \\
 0&0&0&1& \multicolumn{5}{c}
{\raisebox{3ex}[2pt]{\Huge0}}
    \end{array}\right)_{(m+1)\times \frac{(m+1)(m+2)}2}\]
 
and $C$ is a symmetric   $\frac{(m+1)(m+2)}2\times \frac{(m+1)(m+2)}2$ matrix  sketched as, 
 

\[ C=\left(\begin{array}{cccc|ccccc}
e^{r^2}& 0& \cdots &0 & 2r-r^3 &0 &\cdots& 0&r\\
 0&  e^{r^2}&0& \cdots &   &   r  & &\cdots &0 \\
 \vdots & & e^{r^2}&0& \cdots &  &r  &\cdots&0\\
  &   &   &\ddots  & &&& \vdots& \vdots \\ \hline
   2r-r^3& & & &2&&&0&0\\
  \vdots&  & & & &1&&0&0\\
    &    &   & &&&\ddots& \vdots& \vdots \\
        & & & & & &&2&0\\
   r    & \cdots& & &\cdots & && \cdots&1\\
   \end{array}\right)\]
The diagonal elements of $C$ is $\{\underbrace{e^{r^2},\cdots,e^{r^2}}_{\text{$m$}}, \underbrace{2, 1,\cdots,1}_{\text{$m$}},\underbrace{2, 1,\cdots,1}_{\text{$m-1$}}, \cdots, \underbrace{2,1}_{\text{$2$}}, 2,1\}$. For the off diagonal element, it's either $0$ or $2r-r^3$ or $r$.
In fact, we will see that the diagonal elements especially the first $m$ diagonal elements in $C$ are crucial in the  following computations. 

For the short range as $r\to 0$, the matrix $A$ tends to the identity matrix, $B$ tends to a matrix with $(0,\cdots, -1)$ as the first row and a $m\times m$ identity matrix  
in the lower-triangle and $0$ for the rest, and $C$ tends to a diagonal matrix $diag\{ \underbrace{1,1,\cdots, 1}_{\text{$m$}}, 2,1,\cdots, 2 \}$. Hence, as $r\to 0$, $\Lambda=C- B^* A^{-1} B$ tends to a diagonal matrix $diag\{ \underbrace{0,0,\cdots, 0}_{\text{$m$}}, 2,\cdots   \}$ where  the first $m$ elements are $0$. Hence, the Gaussian density  at least degenerates to $\delta_{\xi=0}$ as $r\to0$ for $\xi\in \mathbb C^m$, which implies that the integration \eqref{higherdim} must tend to $0$. 

For the long range as $r\to \infty$, following the same argument as in \S\ref{longrange}, we  change variables $\xi\to e^{r^2/2} \xi$,  then up to  a negligible term $O(r^{-\infty})$, the limit  as $r\to \infty$  is the constant 
 \begin{equation}\label{cm}\begin{split}c_m=  & \frac{ \pi^{m+1}}{\det\tilde\Lambda} \int_{\mathbb C^{\frac{(m+1)(m+2)}{2}}} exp\left\{-(\xi, H_1,\eta)\tilde \Lambda^{-1}\begin{pmatrix}\xi\\ H_1\\ \eta\end{pmatrix} \right\} \|\xi\|^2\\ &\times  \left|\det(H_1^*H_1 -|\eta|^2 I) \right|dV_\xi dV_{H_1}dV_{\eta}  \end{split}  
 \end{equation}where $\tilde\Lambda$ is the diagonal matrix $diag\{\underbrace{1,\cdots,1}_{\text{$m$}}, \underbrace{2, 1,\cdots,1}_{\text{$m$}},\underbrace{2, 1,\cdots,1}_{\text{$m-1$}}, \cdots, \underbrace{2,1}_{\text{$2$}}, 2,1\}$.

\end{document}